\documentclass{amsart}
\usepackage{graphicx}
\usepackage{amscd}
\usepackage{amsmath}
\usepackage{amsfonts}
\usepackage{amssymb}
\theoremstyle{plain}
\newtheorem{theorem}{Theorem}[section]
\newtheorem{corollary}[theorem]{Corollary}
\newtheorem{lemma}[theorem]{Lemma}
\newtheorem{proposition}[theorem]{Proposition}
\theoremstyle{definition}
\newtheorem{example}[theorem]{Example}

\newtheorem{definition}[theorem]{Definition}

\newtheorem{remark}[theorem]{Remark}
\theoremstyle{remark}

\begin{document}
\title{A Schreier type property for modules}

\author{Tiberiu Dumitrescu  and Mihai Epure }

\address{Facultatea de Matematica si Informatica,University of Bucharest,14 A\-ca\-de\-mi\-ei Str., Bucharest, RO 010014,Romania}
\email{tiberiu\_dumitrescu2003@yahoo.com, tiberiu@fmi.unibuc.ro, }

\address{Simion Stoilow Institute of Mathematics of the Romanian AcademyResearch unit 6, P. O. Box 1-764, RO-014700 Bucharest, Romania}\email{epuremihai@yahoo.com, mihai.epure@imar.ro}


\begin{abstract}
\noindent 
We extend to torsion-free modules over integral domains the theory of (pre)-Schreier domains initiated by Cohn and Zafrullah .
\end{abstract}

\maketitle
\section{Introduction}
\noindent

In \cite{Co}, Cohn introduced and studied the {\em Schreier domains}. An (integral) domain $D$ is called a Schreier domain if: 

$(i)$ $D$ is
integrally closed and

$(ii)$  $a|b_1b_2$ with $a,b_1,b_2\in D$, implies $a=a_1a_2$ such that $a_1 | b_1$, $a_2 | b_2$.
\\
A GCD domain (e.g. a UFD) is Schreier  \cite[Theorem 2.4]{Co}.
A polynomial extension of a Schreier domain is a Schreier domain, cf.   \cite[Theorem 2.7]{Co}.

Later on, Zafrullah \cite{Z} studied the domains  satisfying condition $(ii)$ above calling them {\em pre-Schreier domains}.
Subsequently the Schreier domains and their generalizations were studied in \cite{MR}, \cite{DM}, \cite{ADZ}, \cite{DK} and \cite{AD}.

In this paper we   extend the pre-Schreier concept to  modules.  
Say that a torsion-free  module $M$ over the domain $A$ is a {\em PS module (pre-Schreier module)} if 
\begin{center}
$a|bx$ with nonzero $a,b\in A,x\in M$ implies $a=a_1a_2$ in $A$ with $a_1|b$ and $a_2|x$.
\end{center}
Note that this definition is equivalent to Definition \ref{0}, cf. Theorem \ref{1}. Clearly a domain $A$ is a PS module over itself iff $A$ is a pre-Schreier domain.

We  list several results obtained in this paper.
While any torsion-free divisible module is clearly PS, there exists domains, like $\mathbb{R}+X\mathbb{C}[[X]]$, whose PS modules are divisible  hence not finitely generated when nozero (Example \ref{503}). 
A torsion-free module $M$ over a domain $A$ is PS iff $aM\cap Ax$ is a locally cyclic module for all $a\in A$ and $x\in M$ (Theorem \ref{1}).
Recall \cite[page 1899]{Z} that a module   is   {\em locally cyclic} if its  finitely generated  submodules  are contained in   cyclic submodules. A locally cyclic torsion-free module over a domain is flat. 
 
A  submodule $N$ of a PS module $M$ is still PS provided $M/N$ is torsion free (Corollary \ref{101}).
This simple property may be used to produce new PS modules from old ones (e.g. Example \ref{201}).
A domain having a nonzero finitely generated PS ideal is a PS domain
(Proposition \ref{7}). An atomic domain is a UFD provided it has a nonzero finitely generated PS module
(Corollary \ref{604}). Similarly, a Prufer domain is Bezout provided it has a nonzero finitely generated PS module
(Proposition   \ref{44}).
In fact the Bezout domains are exactly the domains whose   torsion-free  modules are  PS (Proposition \ref{44}).
Over a GCD domain, the PS module property is a local property
(Proposition  \ref{43}).
Let $A$ be an atomic  domain and $S$ the multiplicative set of $A$ generated by all nonprime atoms. 
The PS modules over $A$ are exactly the PS modules over the UFD $A_S$
(Theorem \ref{601}).

Any inductive limit or direct sum of PS modules  is a PS module
(Proposition \ref{4} and Theorem  \ref{45}). 
Using Lazard's Theorem \cite[Theorem 1.2]{L}, we show that
a PS module remains PS when tensored by a flat module (Theorem \ref{102}). Consequently, a flat module (e.g. a free one) over a PS domain is a PS module (Corollary \ref{5} and Theorem  \ref{8}). 
Over a weak GCD domain (i.e. a domain where any pair of elements has a maximal common divisor), a direct product of PS modules is a PS module (Theorem \ref{12}). Conversely, a domain $A$ is a GCD domain provided the direct product $A^A$ is a PS  module (Corollary \ref{510}).
An explicit example a PS domain $A$ with $A^\mathbb{N}$ a non-PS module 
 is given in Example \ref{50}. 
Examples of nonflat PS  modules over UFDs are given in 
Examples \ref{201} and \ref{51}.

By \cite[Theorem 2.7]{Co}, a polynomial extension of a Schreier domain is still Schreier. We adapt the proof  \cite[Theorem 2.7]{Co} to obtain a module variant of that result (Theorem \ref{61}).
Over a weak GCD domain, $Hom(M,P)$ is a PS module provided $P$ is PS (Proposition  \ref{62}). In particular, over a GCD domain, the dual of any module is a PS module.
Over a weak GCD domain,  a torsion-free module has a kind of PS envelope (see the text after Proposition \ref{64}). 

Let $A$ be a domain and $M$ a torsion-free $A$-module. If $y=ax$ with 
$a\in A-\{0\}$ and $x,y\in M-\{0\}$, we say that $a,x$ are divisors of $y$ and we write $y/a$ and  $y/x$ for $x$ and $a$ respectively.
Throughout this paper, all rings are commutative and unitary.
For basic results or undefined terminology our references are \cite{FS}, \cite{G} and  \cite{K}.

\section{Basic results}

We give our key definition which is a module extension of the concept of pre-Schreier domain.

\begin{definition}\label{0}  
Let $A$ be a domain and $M$ a torsion-free $A$-module. We say that $M$ is a {\em pre-Schreier module (PS module)}, if every equality 
$$ax=by \mbox{ with } a,b\in A-\{0\},\ x,y\in M-\{0\}$$ 
has a refinement  
$$
\begin{array}{ccccc}
 & \vline & a & x \\
\hline 
 b & \vline & c & d \\
 y & \vline & e & z \\
\end{array}\ \ \ \ \  \mbox{ for some } c,d,e\in A,\ z\in M
$$
in the sense that\ \ 
$a=ce,\ b=cd,\ x=dz,\ y=ez.$
\end{definition}
Note that each one of the elements $c,d,e$ or $z$ above determines the others. This fact enables us to use    phrases like "the element $z\in   M-\{0\}$ gives the refinement"
$$
\begin{array}{ccccc}
 & \vline & a & x \\
\hline 
 b & \vline & (bz)/x & x/z \\
 y & \vline & y/z & z \\
\end{array}
$$
with the meaning that $ax=by$, $z|x,y$ and $x|bz $.

\begin{remark}  \label{21} With notation above, let $t\in A$ be a common divisor of $a$ and $b$. It easily follows that, if $(a/t)x=(b/t)y$ has  a refinement, then so does $ax=by$. A similar remark applies when $t$ is a common factor of $a$ and $y$ (or $x$ and $y$ respectively). Therefore, when trying to find  a refinement of $ax=by$, we may cancel out such common factors.\end{remark}

Clearly  $A$ is a PS module over itself iff $A$ is a pre-Schreier domain in the sense of \cite[page 1895]{Z} called in our paper simply {\em PS domain}. Call an ideal $I$ of  $A$ a {\em PS ideal} if $I$ is a PS module over $A$.  
We give an example of a PS module. Recall that a domain $A$ is {\em atomic} if every non-zero non-unit of $A$ can be written as a finite product of irreducible elements (atoms).

\begin{example} \label{15}  
Let $A$ be an atomic domain and $S$ the multiplicative set of $A$ generated by the non-prime atoms of $A$. Note that the fraction ring $A_S$ is a UFD as every nonzero nonunit of $A_S$ is a product of primes.
We show that $A_S$ is a PS module over $A$. Indeed, let $a(x/t)=b(y/s)$ with $a,b,x,y\in A-\{0\}$ and $s,t\in S.$ 
By Remark \ref{21}, we may cancel out common prime factors 
to reach the case  $a,b\in S$ when we  have the trivial refinement
$
\begin{array}{ccccc}
 & \vline & a & x/t \\
\hline 
 b & \vline & 1 & b \\
 y/s & \vline & a & x/(tb). \\
\end{array}
$

In particular, since the non-prime atoms of $\mathbb{Z}[\sqrt{-3}]$ are $2$, $1\pm \sqrt{-3}$ and $(1 + \sqrt{-3})(1 - \sqrt{-3})=4$, it follows that $\mathbb{Z}[\sqrt{-3},1/2]$ is a PS module over $\mathbb{Z}[\sqrt{-3}]$.
\end{example}

The next result collects some simple facts. A submodule $N$ of an $A$-module $M$ is called a {\em RD-submodule} if $aM\cap N =aN$ for each $a\in A$,
cf. \cite[page 38]{FS}. A direct summand of $M$ is a RD-submodule, cf. \cite[page 39]{FS}.

\begin{proposition} \label{2}  
Let $A$ be a domain and $M$ a torsion-free $A$-module.

$(i)$ If $M$ is divisible, then $M$ is PS. In particular, the quotient field of $A$ is a PS module over $A$.

$(ii)$ If $M$ is PS and $N$ is a RD-submodule of $M$, then $N$ is PS. 

$(iii)$ If $A$ is a Bezout domain, then $M$ is PS.

$(iv)$ Suppose that $M$ is a PS module, $a,b\in A-\{0\}$ and $x\in M-\{0\}$. If $b|ax$ and $a,b$ are coprime (i.e. their gcd is one), then $b|x$. 

$(v)$ If $I$ is a PS ideal of $A$ 
containing a pair of nonzero coprime elements, then $I=A$.

$(vi)$ If $M$ is PS and $S\subseteq A$ is a multiplicative set, 
then $M_S:=M\otimes_A A_S$ is a PS module over the fraction ring $A_S$.
\end{proposition} 
\begin{proof}
Assertion $(i)$ is clear. 
For $(ii)$, let $ax=by \mbox{ with } a,b\in A-\{0\},\ x,y\in N-\{0\}$.
Then a refinement\ \ \  $
\begin{array}{ccccc}
 & \vline & a & x \\
\hline 
 b & \vline & c & d \\
 y & \vline & e & z \\
\end{array}$  \ \ \  has necessarily $z\in N$ since $N$ is an RD-submodule.

For $(iii)$, let $ax=by \mbox{ with } a,b\in A-\{0\},\ x,y\in M-\{0\}$. Dividing by   $gcd(a,b)$, we may assume that $aA+bA=A$. Then $x\in Aby+Abx$, so  we have the refinement 
$
\begin{array}{ccccc}
 & \vline & a & x \\
\hline 
 b & \vline & 1 & b \\
 y & \vline & a & x/b. \\
\end{array}$
\ \ In $(iv)$, if $y=(ax)/b$, then 
$ax=by$  has   essentially  only the refinement  \ \ \ 
$\begin{array}{ccccc}
 & \vline & a & x \\
\hline 
 b & \vline & 1 & b \\
 y & \vline & a & x/b \\
\end{array}$.
For $(v)$ repeat the argument in $(iv)$ with $x=b$ and $y=a$ to get $1=b/b\in I$.
For $(vi)$, let $a'x'=b'y' \mbox{ with } a',b'\in A_S-\{0\},\ x,y\in M_S-\{0\}$. Write $a'=a/s$, $b'=b/s$, $x'=x/s$, $y'=y/s$ with $a,b\in A$, $x,y\in M$ and $s\in S$. Then a refinement 
$
\begin{array}{ccccc}
 & \vline & a & x \\
\hline 
 b & \vline & d & b_1 \\
 y & \vline & a_1 & z \\
\end{array}$
gives the refinement
$
\begin{array}{ccccc}
 & \vline & a' & x' \\
\hline 
 b' & \vline & d & b_1/s \\
 y' & \vline & a_1/s & z \\
\end{array}$.
\end{proof}

\begin{corollary} \label{101}
A  submodule $N$ of a PS module $M$ is still PS provided $M/N$ is torsion free.
\end{corollary}
\begin{proof}
As $N$ is an RD-submodule of $M$,   part $(ii)$ of Proposition \ref{2} applies.  
\end{proof}

\begin{example}\label{503}
Every PS  module $M$ over the ring $A=\mathbb{R}+X\mathbb{C}[[X]]$ is divisible. Indeed, for each $m\in M$, we get $im=(iXm)/X\in  M$ because $X | (iX)(iXm)$ and $X$ is coprime to $iX$. Similarly, from $X | (iX)(im)$, we get $m/X = -i(im)/X$ for each $m\in M$. So $M$ is in fact  a module 
over the quotient field $\mathbb{C}((X))=A[1/X]$ of $A$,  thus $M$ is divisible. Hence zero is the only finitely generated PS module over $A$.
\end{example}

By \cite[Definition page 1899]{Z}, a module $M$ over a domain $A$ is called {\em locally cyclic} if every finitely generated (equivalently two-generated) submodule of $M$  is contained in a cyclic submodule of $M$. If $M$ is also nonzero, then $M$ is a directed union of its nonzero cyclic submodules (copies of $A$),  thus $M$ is flat, cf. \cite[Theorem 1.2]{L}. By \cite[Theorem 1.1]{Z}, $A$ is a PS domain iff $aA\cap bA$ is locally cyclic for all $a,b\in A$.

\begin{theorem} \label{1}  
Let $A$ be a domain and $M$ a torsion-free $A$-module. The following are equivalent.

$(i)$ $M$ is a PS module.

$(ii)$ If $a,b\in A-\{0\},\ y\in M-\{0\}$ and $a|by$, then $a=cd$ for some $c,d\in A$ with $c|b$ and $d|y$.

$(iii)$ If $b\in A-\{0\},\ x,y\in M-\{0\}$ and $x|by$, then $x=az$ for some $a\in A$, $z\in M$ with $a|b$ and $z|y$.

$(iv)$ If $a\in A-\{0\},\ x,y,z\in M-\{0\}$ satisfy 
$a,x | y,z$, then $a,x |u| y,z$ for some $u\in M$. 

$(v)$ $aM\cap Ax$ is a locally cyclic module for
all $a\in A-\{0\}$, $x\in M-\{0\}$.

$(vi)$ $(aM:x)$ is a locally cyclic ideal for
all $a\in A-\{0\}$, $x\in M-\{0\}$.
\end{theorem} 
\begin{proof}
The equivalence of $(i)$, $(ii)$, $(iii)$ can be proven easily. 
For $(i)\Rightarrow (iv)$, let $a\in A-\{0\},\ x,y,z\in M-\{0\}$ such that 
$a,x | y,z$. 
There exists some $x_1\in M-\{0\}$ giving the refinement
$$
\begin{array}{ccccc}
 & \vline & y/x & x \\
\hline 
 a & \vline & (ax_1)/x & x/x_1 \\
 y/a & \vline &  y/(ax_1) & x_1. \\
\end{array}$$
 Then $(ax_1/x)(z/a) =(z/x)x_1$ with $ax_1/x, \ z/x\in A$  and 
 $z/a\in M$, hence 
there exists some $x_2\in M-\{0\}$ giving the refinement
$$\begin{array}{ccccc}
 & \vline & (ax_1)/x & z/a \\
\hline 
 z/x & \vline & (ax_2)/x & z/(ax_2) \\
 x_1 & \vline &  x_1/x_2 & x_2 \\
\end{array}$$
so $a,x | ax_2 | y,z$ because $ax_2|ax_1|y$.
For $(iv)\Rightarrow (i)$, let $ax=by$ with $a,b\in A-\{0\}$ and $x,y\in M-\{0\}$. By $(iv)$, we get $b,x | z | by, bx$ for some $z\in M-\{0\}$,  so $ax=by$ refines to 
$$
\begin{array}{ccccc}
 & \vline & a & x \\
\hline 
 b & \vline & z/x & bx/z \\
 y & \vline & by/z & z/b. \\
\end{array}$$
Next note that $(v)$  is a restatement of $(iv)$, while $(vi)$  is a restatement of $(v)$ because $aM\cap Ax$ is isomorphic to $(aM:x)$.
\end{proof}

Recall that a domain $A$ is an {\em  ACCP domain} if every ascending chain  of principal ideals of $A$ eventually stops. Clearly a Noetherian domain is ACCP.

\begin{corollary} \label{20}  
Let $A$ be an ACCP domain   and $M$ a torsion-free $A$-module. Then $M$ is a PS module iff $(aM:x)$ is a principal ideal 
for all $a\in A-\{0\}$, $x\in M-\{0\}$.
\end{corollary}
\begin{proof}
Use part $(vi)$ of Theorem \ref{1} and the fact that 
in an ACCP  domain every  locally cyclic ideal is principal.
\end{proof}

\begin{example}
Consider the extension $A=\mathbb{Z}[\sqrt{-5}]\subseteq B=\mathbb{Z}[\sqrt{-5},1/2]$. It is well-known that $B$ is a PID hence a PS domain. 
As the ideal $(1+\sqrt{-5})B:_A 1=(3,1+\sqrt{-5})A$ is not principal, $B$ is not a PS module over $A$.
\end{example}

We continue to derive consequences of Theorem \ref{1}.

\begin{proposition}  \label{7}
Let $A$ be a domain with quotient field $K$ and let $M$ be a PS module over $A$.  Then:

$(i)$ $Kx\cap M$ is a locally cyclic PS module for each $x\in M$.

$(ii)$ A rank one PS module is locally cyclic hence flat.

$(iii)$ $A$ is a PS domain provided it has a   finitely generated rank one PS module.
\end{proposition}
\begin{proof}
$(i)$ As $K/(Kx\cap M)$ is a torsion free module, $Kx\cap M$ is a   PS module by Corollary \ref{101}.
Note that 
$Kx\cap M$ is a directed union of  all submodules $A(x/s)\cap M$ with $s\in A-\{0\}$. Moreover $A(x/s)\cap M$ is isomorphic to
$Ax\cap sM$, so it is locally cyclic by part $(v)$ of  Theorem  \ref{1}. Since a directed union of locally cyclic modules is locally cyclic, $Kx\cap M$ is locally cyclic.  Note that
$(ii)$ follows from $(i)$. Finally $(iii)$ follows from $(ii)$ because a nonzero finitely generated locally cyclic module is cyclic, hence   isomorphic to $A$.
\end{proof}

\begin{corollary}  \label{505}
A Noetherian domain is a UFD provided it has a nonzero finitely generated PS module.
\end{corollary}
\begin{proof}
Apply parts $(i)$ and $(iii)$ of Proposition \ref{7}.
\end{proof}

\begin{proposition} \label{511}
Let  $A$ be a domain having a nonprime atom $c$ and let $M$ be a PS module over $A$. Then $M=cM$, so $M=0$ if $M$ is finitely generated.
\end{proposition}
\begin{proof}
As $c$ is nonprime, there exist $a,b\in A-cA$ with $c|ab$. 
Let $x\in M$. As $M$ is a PS module, $c$ is an atom (not dividing $a$ and $b$) and $c|(abx)$, we get $c|bx$, so $c|x$. Thus $M=cM$.
By determinant's trick it follows that $M=0$ if $M$ is finitely generated.
\end{proof}

We get the following improvement of Corollary \ref{505}.

\begin{corollary}  \label{604}
An atomic domain is a UFD provided it has a nonzero finitely generated PS module.
\end{corollary}

\begin{proposition} \label{44}
For a domain $A$ the following are equivalent.

$(i)$  $A$ is a Bezout domain.

$(ii)$ $A$ is a Prufer domain and has a nonzero finitely generated PS module.

$(iii)$ All torsion-free $A$-modules are PS.
\end{proposition}
\begin{proof}
$(iii) \Rightarrow (ii)$. By part  $(ii)$ of Proposition \ref{7}, all ideals of $A$ are flat, hence $A$ is a Prufer domain by \cite[Theorem 9.10]{FS}.
$(ii) \Rightarrow (i)$. 
Since every  finitely generated torsion free module is isomorphic to a direct sum of finitely generated ideals (cf. \cite[Corollary 3.7.11]{WK}), we may assume that $A$ has a rank one finitely generated PS module, cf. part $(ii)$ of Proposition \ref{2}. By part $(iii)$ of Proposition \ref{7}, $A$ is a PS domain, so $A$ is a Bezout domain cf. \cite[Theorem 2.8]{Co}.
Finally $(i) \Rightarrow (iii)$ is part $(iii)$ of Proposition \ref{2}.
\end{proof}

Over a GCD domain, the PS module property is   local.

\begin{proposition}  \label{43}
Let $A$ be an GCD domain and $M$ a torsion-free $A$-module. Then $M$ is PS iff $M_P$ is a PS module over $A_P$ for each maximal ideal $P$ of $A$.
\end{proposition}
\begin{proof}
$(\Rightarrow)$ follows from part $(vi)$ of Proposition \ref{2}.
$(\Leftarrow).$ Suppose that $b|ax$  with $a,b\in A-\{0\}$ and $x\in M-\{0\}$.  Dividing by $gcd(a,b)$, we may assume that $aA\cap bA=abA$ (see Remark \ref{21}). Let $P$ be a maximal ideal of $A$. Then $aA_P\cap bA_P=abA_P$, so $a,b$ are coprime in $A_P$. As $M_P$ is a PS module over $A_P$, we get $x/b\in M_P$, cf. part $(iv)$ of Proposition \ref{2}. As this fact happens for each $P$, we get  $x/b\in \cap_{P} M_P = M$.
\end{proof}

\begin{remark}
If $A$ is a Dedekind domain which is not a PID (e.g. $\mathbb{Z}[\sqrt{-5}]$), then $A_P$ is a PS domain for all maximal ideals $P$ of $A$, but $A$ is not a PS domain.
\end{remark}

\section{Scalar restriction}

The following scalar restriction result extends Example \ref{15}.

\begin{theorem}\label{18}  
Let $A\subseteq B$ be an extension of domains such that $bB\cap A$ is a locally cyclic ideal  of $A$ for each $b\in B$.
If $M$ is  a PS module over $B$, then $M$ is  PS as an $A$-module.
\end{theorem}
\begin{proof}
Let $a\in A-\{0\}$ and $x\in M-\{0\}$. 
By part $(vi)$ of Theorem \ref{1}, it suffices to show that $(aM:_A x)$ is a locally cyclic ideal of $A$.  Let $a_1,a_2\in (aM:_A x)$. As $(aM:_B x)$  is locally cyclic (cf. part $(vi)$ of Theorem \ref{1}), we have $a_1,a_2\in bB$ for some $b\in (aM:_B x)$. By hypothesis $bB\cap A$ is locally cyclic, so $a_1,a_2\in  cA\subseteq bB$ for some $c\in A$.
 We get $c\in  bB\cap A \subseteq A\cap (aM:_B x)=(aM:_A x)$.
\end{proof}


Let $A$ be a domain and $S$ a saturated multiplicative subset of $A$.
Recall  that $S$ is a {\em splitting set} if 
$aA_S\cap A$ is a principal ideal for each $a\in A$,
cf. Definition 2.1 and Theorem 2.2 in \cite{AAZ}.
We state an  immediate consequence of Theorem \ref{18}.

\begin{corollary} \label{19}  
Let $A$ be  domain, $S$ a splitting multiplicative subset of $A$     and $M$   a PS module over $A_S$. Then $M$ is a PS module over $A$.
\end{corollary}

Let $A$ be an atomic  domain and $S$ the multiplicative set of $A$ generated by all nonprime atoms. The fraction ring $A_S$ is a UFD because each nonzero nonunit of $A_S$ is a product of prime elements.

\begin{theorem}\label{601}  
With notation above, the following assertions hold.

$(i)$ If $M$ is  a PS module over $A$, then $M$ is PS module over $A_S$.

$(ii)$ Conversely, if $N$ is  a PS module over $A_S$, then $N$ is  PS  module over $A$.
\\ Hence the PS modules over $A$ are exactly the PS modules over the UFD $A_S$.
\end{theorem}
\begin{proof}
$(i)$. By Proposition \ref{511}, $M$ is a module over $A_S$. Hence 
$M=M\otimes_A A_S$ is a PS module over $A_S$ by part $(vi)$ of Proposition \ref{2}. $(ii)$. $S$ is a splitting set because $pA_S \cap A=pA$ for each prime element $p\in A$ and every nonzero nonunit of $A_S$ is, up to unit multiplication, a product of such primes. Apply Corollary \ref{19}.
\end{proof}

\begin{remark}
Using Remark \ref{21}, it is easy to see that a torsion-free module $M$ over a UFD $A$ is a PS module iff for each prime element $p$ of $A$,

$p|ax$ with $a\in A$ and $x\in M$ implies $p|a$ or $p|x$.

\end{remark}

\begin{remark}
Let $A$ be an atomic  domain whose atoms are   nonprime. It follows from Theorem \ref{601} and its proof that any PS module over $A$ is divisible.
See also Example \ref{503}.
\end{remark}

We recall the following fact \cite[Corollary 6.6 and Proposition 6.8]{F}.

\begin{proposition} \label{602}
Let $A$ be a UFD. Then every flat overring of $A$ is a fraction ring of $A$. 
\end{proposition}

\begin{proposition} 
Let $A$ be an atomic  domain and $S$ the multiplicative set of $A$ generated by all nonprime atoms. Then the overrings of $A$ which are PS modules over $A$ are the fraction rings $A_T$ with $T$ a multiplicative set containing $S$.
\end{proposition}
\begin{proof}
Let $B$ be an overring  of $A$ which is a PS module  over $A$.
By Theorem \ref{601} and part $(ii)$ of Proposition \ref{7}, $B$ is a flat overring of $A_S$, so $B$ is a fraction ring of $A_S$, cf. Proposition \ref{602}. Conversely, let $T$ be a multiplicative set of $A$ containing $S$.
Then $A_T$ is a PS module over $A_S$, cf. Corollary \ref{8}.
Apply Theorem \ref{601}.
\end{proof}

\section{Flat modules}

\begin{proposition} \label{4}  
Any inductive limit of PS modules  is a PS module.
\end{proposition} 
\begin{proof}
Let $A$ be a  domain, $S=\{f_{ij}:M_{i}\rightarrow M_j\ |\ i\leq j,\ i,j\in I \}$  an inductive system of PS modules over $A$ and let $M$ be the limit of $S$ with canonical maps $\{ f_i:M_i\rightarrow M\ |\ i\in I\}$. As every $M_i$ is torsion-free so is $M$. 
Suppose that $ax=by$ with $a,b\in A-\{0\},\ x,y\in M-\{0\}$.
Then $ax'=by'$ for some $i\in I$, $x',y'\in M_i$ with $f_i(x')=x$, $f_i(y')=y$. As $M_i$ is a PS module, we have a refinement  
$$
\begin{array}{ccccc}
 & \vline & a & x' \\
\hline 
 b & \vline & c & d \\
 y' & \vline & e & z' \\
\end{array}  \ \ \ \mbox{ with } c,d,e\in A,\ z'\in M_i.$$
which gives the refinement
$
\begin{array}{ccccc}
 & \vline & a & x \\
\hline 
 b & \vline & c & d \\
 y & \vline & e & f_i(z'). \\
\end{array}
$
\end{proof}

\begin{theorem} \label{45}
Let $A$ be a   domain. A direct sum of torsion-free $A$-modules is PS iff each term is PS.
\end{theorem}
\begin{proof} 
 $(\Rightarrow)$  follows from part $(ii)$ of Proposition \ref{2}.
$(\Leftarrow)$ Using an induction and Proposition \ref{4}, it suffices to consider 
a direct sum $M\oplus N$ of  two PS modules.   
Let $a,b\in A-\{0\}$ and $(x,y)\in M\oplus N$ such that
$a|b(x,y)$. As $M$ is PS, we have $c|a,b$ and $(a/c)|x$ for some $c\in A-\{0\}$. Set $a'=a/c$ and $b'=b/c$. Then $a'|b'y$. As $N$ is PS, we have 
$c'|a',b'$ and $(a'/c')|y$ for some $c'\in A-\{0\}$. It follows that
$(cc')|a,b$ and $(a/(cc'))|(x,y)$.
\end{proof}

\begin{corollary} \label{5}  
Any free module over a PS domain is a PS module.
\end{corollary}

\begin{theorem} \label{102}
If $A$ is a domain, $M$ a PS $A$-module and $P$ a flat $A$-module then, $M\otimes_A P$ is PS module.
\end{theorem}
\begin{proof} 
By Lazard's Theorem \cite[Theorem 1.2]{L}, $P$ is an inductive limit of free modules, so $M\otimes_A P$ is an inductive limit of direct sums of copies of $M$, hence $M\otimes_A P$ is PS, cf. Proposition \ref{4} and Theorem \ref{45}.
\end{proof}

\begin{corollary}  \label{8}  
Any flat module over a PS domain is a PS module.
\end{corollary} 
\begin{proof}
Apply Theorem \ref{102} for $M=A$.
\end{proof}

\begin{example} \label{201} 
Over a UFD, a PS module is not necessarily flat. 
Let $A$ be a PS domain and $I$ an ideal of flat dimension $\geq 2$. 
Consider a short exact sequence of the form 
$0\rightarrow M\rightarrow F\rightarrow I\rightarrow 0$
where $F$ is a flat module. Then  $M$  is a PS module (Corollary  \ref{101} and Corollary \ref{8}) but not  flat. For a specific example, we may take $A=K[X,Y,Z]$ for some field $K$,  $I=(X,Y,Z)$ and  $M$ the first syzygy module of $I$.
\end{example} 

\begin{corollary} \label{14}  
The flat ideals of a PS domain are locally cyclic. 
\end{corollary}
\begin{proof}
Apply  Corollary \ref{8} and  Proposition \ref{7}.
\end{proof}

From Corollary \ref{14}, we recover the well-known fact that the flat ideals of a UFD are principal.
According to \cite[page 5]{AAZ0}, a domain $A$ is a {\em weak GCD domain}  if each pair $a$, $b\in A-\{0\}$ has a maximal common divisor $d$, that is with $gcd(a/d,b/d)=1$.  
The ACCP domains and the GCD domains are examples of weak GCD domains.

\begin{theorem} \label{12}  
Over a weak GCD domain, a direct product of PS modules is a PS module. 
\end{theorem}
\begin{proof}
Let $A$ be a weak GCD domain, $(M_i)_{i\in I}$ a family of PS modules and $M$ their direct product. 
 Let $a(x_i)_i=b(y_i)_i$ with $a,b\in A-\{0\}$ and $x_i,y_i\in M_i-\{0\}$. Dividing by a maximal common divisor of $a$ and $b$, we may assume that $a$ and $b$ are coprime.
As $ax_i=by_i$, we get   $b|x_i$ (cf. part $(iv)$ of Proposition \ref{2}), thus giving the refinement
$
\begin{array}{ccccc}
 & \vline & a & (x_i)_i \\
\hline 
 b & \vline & 1 & b \\
 (y_i)_i & \vline & a & (x_i/b)_i \\
\end{array}$.
\end{proof}

\begin{corollary} \label{510}
A domain $A$ is a GCD domain iff the direct product $A^A$ is a PS $A$-module.
\end{corollary}
\begin{proof}
$(\Rightarrow)$  follows from Theorem \ref{12}.
$(\Leftarrow)$  Pick $a,b\in A-\{0\}$ and let $F$ be the set of all common multiples of $a$ and $b$. 
Adding some zero components, we may consider that $x=(f)_{f\in F}$ is an element of $A^A$. Then $a,b|x$, so $b | a(x/a)$. 
As $A^A$ is PS,  $c|a,b$ and $(b/c)| (x/a)$ for some $c\in A$. Then $a,b | (ab/c) | x$, so $ab/c$ is the least common multiple of $a$ and $b$.
\end{proof}

Let $A$ be a PS domain which is not a GCD domain (e.g. \cite[Example 4.5]{Z}). By the previous corollary, $A^A$ is not a PS module, so it is not flat by Corollary  \ref{8}.
Next we exhibit an example of this kind where $A^{\mathbb{N}}$ is not a PS module. 

\begin{example} \label{50}
Consider the  domain  $D=\mathbb{Q} + P_P$ where $S=\{ X^r\ |\ r\in\mathbb{Q}, r\geq 0 \}$ and $P=\{ f\in \mathbb{R}[S]\ |\ f(0)=0\}$ of  \cite[Example 4.5]{Z}. 
We show that    $M=D^{\mathbb{N}}$ is not a PS module over $D$.
By \cite[Example 4.5]{Z}, $D$ is a PS domain which is not 
integrally closed.  Note that  $\sqrt{2}X\in D$ divides the product $Xf$ where $f=(X^{1/n})_{n\geq 1}\in M$. Suppose that  $\sqrt{2}X=ab$ with $a,b\in D$ such that $a|X$ and $b|f$. Then, up to unit multiplication, we get $a=cX^r$, $b=dX^q$  for some $c,d\in \mathbb{R}$, $r,q\in\mathbb{Q}$ with $0< r,q < 1$, $r+q=1$. Then $q\leq 1/n$ for each $n$  
which is a contradiction. Thus $M$ is not a PS module.
\end{example}

\begin{example} \label{51}
By \cite[Example 6.2]{CG}, there exists a UFD $A$ which is not coherent.  By Chase's Theorem \cite[Theorem 2.1]{Ch}, some direct product $A^I$ of copies of $A$ is not a flat $A$-module. By Theorem \ref{12}, $A^I$ is a PS module.
\end{example}

\section{Nagata type theorem}

Let $A$ be a domain and $M$ a torsion-free $A$-module.
Say that an element $a\in A-\{0\}$ is $M$-primal if 
whenever $a|by$ with $b\in A-\{0\}$ and $y\in M-\{0\}$, we have a factorization $a=a_1a_2$ with $a_1,a_2\in A$ such that $a_1|b$ and $a_2|y$.  Note that $a$ is $A$-primal iff it is primal in the sense of \cite[page 254]{Co}. We extend \cite[Theorem 2.6]{Co} to modules.

\begin{theorem} \label{16}  
Let $A$ be a domain and $M$ a torsion-free $A$-module.
Let $S$ a saturated multiplicative subset of $A$ consisting of elements which are both $A$-primal and $M$-primal.
If $M_S$ is a  PS  module over $A_S$, then $M$ is a PS module.
\end{theorem}
\begin{proof}
We adapt the proof of \cite[Theorem 2.6]{Co}.
Let $ax=by$ with $a,b\in A-\{0\}$, $x,y\in M-\{0\}$. Since $M_S$ is a  PS over $A_S$, we have  a 
refinement
$\begin{array}{ccccc}
 & \vline & a & x \\
\hline 
 b & \vline & a_1 & b_1/s \\
 y & \vline & a_2 & z/t \\
\end{array}$
with $a_1,a_2,b_1\in A$, $s,t\in S$ and $z\in M$, which gives the refinement
$\begin{array}{ccccc}
 & \vline & a & stx \\
\hline 
 sb & \vline & a_1 & b_1 \\
 ty & \vline & a_2 & z \\
\end{array}$.
As $S$ is saturated, $s$ is $M$-primal and $s|b_1z$, we can write $s=s_1s_2$ with $s_1,s_2\in S$, $s_1|b_1$ and $s_2|z$. Since $s_2$ is $A$-primal and  $s_2|(b_1/s_1)a_1$, we can write $s_2=s_3s_4$ with $s_3,s_4\in S$, 
$s_3|(b_1/s_1)$ and $s_4|a_1$. We get the refinement
$\begin{array}{ccccc}
 & \vline & a &  tx \\
\hline 
 b & \vline & a_1/s_4 & b_1/(s_1s_3) \\
 ty & \vline & a_2s_4 & z/s_4 \\
\end{array}.$
Set $c_1=a_1/s_4$, $d=b_1/(s_1s_3)$, $c_2=a_2s_4$, $z'=z/s_4$.
We repeat the argument above starting from
$\begin{array}{ccccc}
 & \vline & a &  tx \\
\hline 
 b & \vline & c_1 & d \\
 ty & \vline & c_2 & z' \\
\end{array}.$
We can write $t=t_1t_2$ with $t_1,t_2\in S$, $t_1|d$ and $t_2|z'$. Since $t_1$ is $M$-primal and  $t_1|c_2(z'/t_2)$, we can write $t_1=t_3t_4$ with $t_3,t_4\in S$,  $t_3|c_2$ and $t_4|(z'/t_2)$.
We get the refinement
$\begin{array}{ccccc}
 & \vline & a &   x \\
\hline 
 b & \vline & c_1t_3& d/t_3 \\
  y & \vline & c_2/t_3 & z'/(t_2t_4)
\end{array}$ \ \ thus completing the proof.
\end{proof}


Let $R$ be a ring $M$ an $R$-module, $f\in R[X]$ and $g\in M[X]$.
The content of $f$ is the ideal of $R$ generated by the coefficients of $f$. Similarly, the content of $g$ is the submodule of $M$ generated by the coefficients of $g$. The following  theorem is well-known, see for instance \cite[Theorem 1.7.16]{WK}.

\begin{theorem} (Dedekind-Mertens) \label{17}  
Let $R$ be a ring, $M$ an $R$-module, $f\in R[X]$ and $g\in M[X]$. There is some $m\geq 1$ such that
$c(f)^{m}c(fg)=c(f)^{m+1}c(g).$
\end{theorem}

\begin{lemma} \label{40}
Let $A$ be  a PS domain and $M$ a PS module over $A$.
If $a_1,...,a_n,c\in A-\{0\},\ x_1,...,x_m\in M-\{0\}$ satisfy 
$a_1,...,a_n | c,x_1,...,x_m$, then $a_1,...,a_n | d |c,x_1,...,x_m$ for some $d\in A$.
\end{lemma}
\begin{proof} 
The general case can be obtained by induction from the basic case $n=2$, $m=1$, so it suffices to  settle this case.
Assume that $a,b | c,x$ with $a,b,c\in A-\{0\}$ and $x\in M-\{0\}$.
There exists some $b_1\in A-\{0\}$ giving the refinement
$$
\begin{array}{ccccc}
 & \vline & c/b & b \\
\hline 
 a & \vline & (ab_1)/b & b/b_1 \\
 c/a & \vline &  c/(ab_1) & b_1. \\
\end{array}$$
 Then $(ab_1/b)(x/a) =b_1(x/b)$ with $ab_1/b\in A$  and 
 $x/a, x/b\in M$, hence 
there exists some $b_2\in A-\{0\}$ giving the refinement
$$\begin{array}{ccccc}
 & \vline & (ab_1)/b & x/a \\
\hline 
 b_1 & \vline &  b_1/b_2 & b_2 \\
  x/b& \vline & (ab_2)/b & x/(ab_2) \\
\end{array}$$
so $a,b | ab_2 | c,x$ because $ab_2|ab_1|c$.
\end{proof}

Recall that a {\em Schreier domain} is an integrally closed PS domain. 
Let $M$ be a torsion-free module over a domain $A$. In the spirit of \cite{Re}, say that $M$ is integrally closed if 
$$M=\cap \{ VM\ | \ V \mbox{ valuation overring of } A\}.$$

\begin{theorem} \label{61}
If $A$ is a Schreier domain and $M$ an integrally closed PS module over $A$, then $M[X]=:M\otimes_A A[X]$ is a PS module over $A[X]$.
\end{theorem}
\begin{proof} 
By \cite[Theorem 2.7]{Co}, every element of $S=A-\{0\}$ is $A[X]$-primal. 
We show that every  $d\in S$ is $M[X]$-primal.
Suppose that $d|fg$ with $f\in A[X]-\{0\}$ and $g\in M[X]-\{0\}$.
Let $V$ be a valuation overring of $A$. Then 
by Theorem \ref{17}, $c(f)^{l}c(fg)V=c(f)^{l+1}c(g)V$ for some $l\geq 1$, so 
$c(fg)V=c(f)c(g)V$ because $c(f)V$ is principal. 
Let  
$a_0,...,a_n$ and $y_0,...,y_m$ be the nonzero coefficients of $f$ and $g$ respectively. Set $a=a_0a_1\cdots a_n$.
Then $d|a_iy_j$ in $MV$, hence $d|a_iy_j$ in $M$ since is integrally closed. 
As $A$ and $M$ are PS, Lemma \ref{40} shows there exists some $q\in A$ with
$$a, (ad)/a_0,\ ...,\ (ad)/a_n\ |\ q\ |\  ad,\ ay_0,\ ...,\ ay_m.$$
We see that $(d/q)|f$ and $q|g$, so $d$ is $M[X]$-primal.
Note that $A_S=K$ is the quotient field of $A$ and $M[X]\otimes K$ is a PS module over $K[X]$ because $K[X]$ is a PID.
By Theorem \ref{16}, $M[X]$ is a PS module over $A[X]$.
\end{proof}

\section{Odds and ends}

\begin{proposition} \label{62}
If $A$ is a domain and $M,P$ are $A$-modules such that $P$ is PS, then $Hom(M,P)$ is a PS module, provided one of the following two conditions holds.

$(i)$  $M$ is finitely generated.

$(ii)$ $A$ is a weak GCD domain.
\\ In particular, if $A$ is an GCD domain, then the dual of any $A$-module is PS.
\end{proposition} 
\begin{proof}
Clearly $Hom(M,P)$ is torsion-free.
$(i)$ From a short exact sequence of the form $0\rightarrow K\rightarrow A^n\rightarrow M\rightarrow 0$, we get the exact sequence
$$0\rightarrow Hom(M,P)\rightarrow Hom(A^n,P)=P^n\rightarrow Hom(K,P)$$ with $Hom(K,P)$ torsion free,
so $Hom(M,P)$ is a PS module by Corollary \ref{8} and Corollary  \ref{101}.
For $(ii)$, repeat the argument above invoking Theorem \ref{12} instead of Corollary \ref{8}.
\end{proof}

We connect our PS module concept to factorial and factorable modules. These notions were introduced by A.-M. Nicholas in \cite{N}, \cite{N2} and refined by D. L. Costa \cite{C}, C.-P. Lu \cite{Lu} and D. D. Anderson, S. Valdes-Leon \cite{AV}. More recent work on this subject was done by G. Angermüller in \cite{A}. First, we recall some definitions.

\begin{definition} 
Let $A$ be a domain and $M$ a nonzero torsion-free $A$-module. 

$(i)$ An element $x\in M-\{0\}$ is an irreducible element (atom) if every scalar
$a\in A$ dividing $x$ is a unit.

$(ii)$ $M$ is atomic if every $x\in M-\{0\}$ is divisible by some atom.

$(iii)$ An element $x\in M-\{0\}$ is primitive if the factor module $M/Ax$ is torsion-free.

$(iv)$  $M$ is factorable if $M$ is atomic and all atoms are primitive.

$(v)$  $M$ is factorial if $A$ is a UFD and $M$ is factorable.
\end{definition}

\begin{proposition}  \label{100}
Let $A$ be a  domain and  $M$ a torsion-free module over $A$.

$(i)$ If $M$ is atomic and PS, then $M$ is factorable.

$(ii)$ If $A$ is PS  and $M$ is factorable, then $M$ is PS. 

$(iii)$ If $M$ is factorial, then $M$ is PS. 
\end{proposition}
\begin{proof} 
$(i)$ 
By \cite[Proposition 2]{A}, it suffices to show that every irreducible element $x$ of $M$ is primitive. Let $x|by$ with 
$b\in A-\{0\}$ and $y\in M-\{0\}$. By part $(iii)$  of
Theorem \ref{1}, we have $x=az$ for some $a\in A$, $z\in M$ with $a|b$ and $z|y$. As $x$ is irreducible, $a$ is a unit, hence $x|y$. 
So $x$ is primitive.

$(ii)$ 
Let $ax=by$ with $a,b\in A-\{0\}$, $x,y\in M-\{0\}$. 
As $M$ is factorable, we can write $x=a'z$ with $a'\in A$ and 
$z\in M$ primitive. From $z|by$, we get $z|y$, so $y=b'z$ for some $b'\in A$. Cancel $z$ to get $aa'=bb'$. As $A$ is PS, we have a
refinement 
$\begin{array}{ccccc}
 & \vline & a & a' \\
\hline 
 b & \vline & c & b_1 \\
 b' & \vline & a_1 & d \\
\end{array}$
which gives the refinement 
$\begin{array}{ccccc}
 & \vline & a & x \\
\hline 
 b & \vline & c & b_1 \\
 y & \vline & a_1 & dz \\
\end{array}.$

$(iii)$ If $M$ is factorial, then $A$ is a UFD and $M$ is factorable \cite[Proposition 3]{A}, so $(ii)$ applies.
\end{proof}

\begin{remark} 
If $A$ is a PS domain which is not a GCD domain (e.g. \cite[Example 4.5]{Z}), then $A^2$ is a PS $A$-module but not factorable, cf. \cite[Corollary 7]{A}. If $A$ is a non-PS domain, then $A$ over itself is a factorable non-PS module.
\end{remark}

\begin{proposition} \label{64}
Let $A$ be a weak GCD domain  and $M$ a torsion-free module over $A$. Then an intersection of PS submodules of $M$ is PS.
\end{proposition} 
\begin{proof}
Let $F$ be a family of PS submodules with intersection $N$ and let $a|by$ with $x,y\in N$. We may assume $(a,b)=1$. Then $y/a\in N$.
\end{proof}

Let $A$ be a weak GCD domain, $M$ a PS module over $A$ and $N$ a submodule of $M$. Let $P$ be the intersection of all PS submodules of $M$ containing $N$. Then $P$ is PS by Proposition \ref{64}. We give another description of $P$.
For a submodule $Q$ of $M$, let $Q'$ be the submodule of $M$ generated by all vectors $x/a$ where $x\in Q$ and $bx\in aQ$ for some coprime scalars $a,b\in A-\{0\}$. 
Iterating this prime operation $Q\mapsto Q'$, we get the ascending sequence of submodules of $M$ 
$$N_0=N\subseteq N_1=N_0'\subseteq ... \subseteq N_n \subseteq 
N_{n+1} = N_n'\subseteq ... 
$$
whose union we denote by $S$. We show that $P=S$. To prove that $S$ is PS, let $ax=by$ with $a,b\in A-\{0\}$ and $x,y\in S-\{0\}$. Dividing by a maximal common divisor of $a$ and $b$, we may assume that $a$ and $b$ are coprime. As $x,y\in N_n$ for some $n$, we get that $y/a\in N_{n+1}$ 
 thus giving the refinement
$
\begin{array}{ccccc}
 & \vline & a & x \\
\hline 
 b & \vline & 1 & b \\
 y & \vline & a & y/a \\
\end{array}$. Therefore $S$ is a PS module, hence $P\subseteq S$. The other inclusion is obvious.

\end{document}